\documentclass[11pt,reqno]{amsart}
\usepackage {amsmath,amsfonts,enumerate, epsfig, epsf, color, indentfirst,  graphicx,euscript, amssymb,amsthm,amscd}

\usepackage[dvips]{psfrag}

\usepackage[latin1]{inputenc}

\makeindex

\usepackage{amssymb}
\usepackage[latin1]{inputenc}
\usepackage{ epsfig,epsf}
\usepackage{enumerate}
\usepackage{indentfirst}
\usepackage{euscript}
\usepackage{graphicx}
\usepackage{amsmath}
\usepackage{amsfonts}
\usepackage{amssymb}
\makeindex

\makeindex
 \theoremstyle{plain}
\newtheorem{proposition}{Proposition}

\theoremstyle{definition}

\newtheorem{theorem}{Theorem}

\def\alf{{\mathcal L}_{\alpha,1}}
\def\als{{\mathcal L}_{\alpha,2}}
\def\aff{{\mathcal A}_{\alpha,1}}
\def\afs{{\mathcal A}_{\alpha,2}}

  \title[ Dense Asymptotic   Lines on Tori  Embedded  in  $\mathbb S^3$
]{Tori Embedded in $\mathbb S^3$ with Dense   Asymptotic   Lines}

 \author{ Ronaldo Garcia and Jorge Sotomayor }

\begin{document}
 \maketitle

   \begin{abstract} In this paper are   given examples of
tori  $\mathbb T^2$  embedded  in $\mathbb S^3$
with  all their asymptotic lines
 dense.
 \end{abstract}

\section{ Introduction}

Let $\alpha:\mathbb M \to\mathbb S^3$ be
an immersion of class
 $ C^{r }, r\geq 3,$
of a smooth, compact and oriented two-dimensional
manifold $\mathbb M  $ into  the three dimensional sphere
${\mathbb S}^{3}$ endowed with the canonical inner product $<.,.>$
of $\mathbb R^4$.

 The {\em Fundamental Forms} of $\alpha $ at
a point $p$ of ${\mathbb M}$ are the symmetric bilinear forms on
${\mathbb T}_p \mathbb M$ defined as follows (Spivak 1999):

 $$\aligned I_{\alpha }(p;v,w)=& <D\alpha
(p;v),D\alpha (p;w)>,\\
 II_{\alpha}(p;v,w)=& <-DN_{\alpha}(p;v),D\alpha(p;w)>.\endaligned
$$

 \noindent Here,  $N_{\alpha }$ is the positive unit normal of the immersion $\alpha$ and $<N_\alpha,\alpha>=0$.

 Through every point $p$ of the {\em hyperbolic region}
 ${\mathbb H}_{\alpha }$ of the immersion $\alpha,
$ characterized by the condition that the extrinsic Gaussian
Curvature ${\mathcal K}_{ext}$
 $=\det(D{ N}_{\alpha})$ is negative,
pass
two transverse asymptotic lines of $\alpha$, tangent
to the two asymptotic directions through $p$. Assuming $r\geq 3$
this follows from the usual existence and uniqueness theorems on
Ordinary Differential Equations. In fact, on ${\mathbb H}_{\alpha
}$ the local line fields are defined by the kernels
 \;\; $\alf$,\;\;  $\als$\;\;  of the smooth one-forms  $
\omega_{ \alpha,1 },$ \; $ \omega_{ \alpha,2 }  $\;  which
locally split\;\,
 $II_{\alpha }$
 as the product of \,  $\omega_{\alpha,1}$
 and \,  $\omega_{\alpha,2}$.

The forms $\omega_{\alpha,i}$ are locally defined up to a non
vanishing factor and a permutation of their indices. Therefore,
their kernels and integral foliations are locally well defined
only up to a permutation of their indices.

Under the orientability hypothesis imposed on ${\mathbb M}, $ it
is possible to globalize, to the whole ${\mathbb H}_{\alpha }, $
the definition of the line fields $\alf,$ $\als$
 and of the choice of an ordering
between them, as established  in (Garcia and Sotomayor 1997)   and (Garcia et al. 1999).

 These two line fields, called the {\em asymptotic line fields} of
$\alpha $, are of class $C^{r-2}$ on ${\mathbb H}_{\alpha }$;
 they are distinctly defined together with the ordering between
them given by the subindexes $\{1,2\}$ which define their {\em
orientation ordering}: {\em ``1"}
for the {\em first asymptotic
line field} $\alf,$ {\em ``2"} for the {\em second asymptotic line
field} $\als$.

The {\em asymptotic foliations} of $\alpha $ are the integral
foliations $\aff$ of $ \alf$ and $\afs $ of $\als$; they fill out
the hyperbolic region ${\mathbb H}_{\alpha }$.

In a local chart $(u,v)$ the asymptotic directions of an immersion
$\alpha$ are defined by the implicit differential equation
$$II=e du^2+ 2f dudv+gdv^2=0.$$

In $\mathbb S^3$, with  the second fundamental form   relative to
the normal vector $N=\alpha\wedge \alpha_u\wedge \alpha_v$, it
follows that:
 $$ \text{\small $ e=\frac{  \det[\alpha, \alpha_u,\alpha_v,\alpha_{uu}]}{
\sqrt{   EG -F^2}}, \;
 f=\frac{  \det[\alpha, \alpha_u,\alpha_v,\alpha_{uv}]}{\sqrt{EG-F^2}},\;
g=\frac{ \det[\alpha,\alpha_u,\alpha_v,\alpha_{vv}]}{\sqrt{ EG-F^2
}}.$}  $$

There is a considerable difference between the cases of surfaces
in the  Euclidean and in the Spherical spaces. In $\mathbb R^3$
the asymptotic lines are never globally defined for immersions of
compact, oriented surfaces.
 This is due to the fact that in these surfaces there are always elliptic points, at which ${\mathcal K}_{ext} > 0$
 (Spivak  1999,   Vol. III, chapter 2, pg. 64).

The study of asymptotic  lines on  surfaces $\mathbb M$  of
$\mathbb R^3$ and $\mathbb S^3$ is a classical subject of
Differential Geometry. See  (do Carmo  1976,  chapter 3),
(Darboux  1896,  chapter II),  (Spivak 1999, vol. IV,
chapter 7, Part F)  and (Struik  1988, chapter 2).

 In (Garcia and Sotomayor  1997)   and (Garcia et al. 1999)
  ideas coming from the Qualitative Theory of Differential Equations and Dynamical Systems such as
 Structural Stability  and Recurrence were introduced into the subject of Asymptotic Lines.
 Other differential equations of Classical
 Geometry have been considered in  (Gutierrez and Sotomayor 1991); a recent survey can be found in
  (Garcia and Sotomayor  2008a).

 The interest  on  the  study  of foliations with dense
 leaves goes back to Poincar\'e, Birkhoff, Denjoy, Peixoto, among others.

In $\mathbb S^3$ the asymptotic lines can be globally defined, an
example is the Clifford torus, ${\mathcal C}=\mathbb S^1(r)\times
\mathbb S^1(r)\subset \mathbb S^3$, where $\mathbb
S^1(r)=\{(x,y)\in \mathbb R^2:\; x^2+y^2=r^2\}$ and
$r=\sqrt{2}/2$. In $\mathcal C$
 all asymptotic lines are closed curves, in fact,   Villarceau   circles.
  (See Villarceau 1848)
and  illustration in Fig. \ref{fig:villa}.

An asymptotic line $\gamma$ is called {\it  recurrent} if it is
contained in the hyperbolic region and     $\gamma\subseteq
L(\gamma)$, where $L(\gamma)=\alpha(\gamma)\cup \omega(\gamma)$ is
the limit set of $\gamma$, and it is called {\it dense } if
$ L(\gamma) =\mathbb M$.

In this paper is given an example of an embedded torus
(deformation of the Clifford torus) with both asymptotic
foliations  having all their  leaves dense.

  \section{Preliminary Calculations}

In this section will be obtained the variational equations of a
quadratic differential equation to be  applied in the analysis
 in Section  \ref{sc:m1}.

\begin{proposition}\label{prop:eqee} Consider a one parameter family of  quadratic differential equations of the form

\begin{equation} \aligned \label{eq:qe} a&(u,v,\epsilon) dv^2+ 2b(u,v,\epsilon) dudv+ c(u,v,\epsilon)du^2=0,\\
a&(u,v,0)=c(u,v,0)=0, \;\; b(u,v,0)=1.
\endaligned
 \end{equation}

Let $v(u,v_0,\epsilon)$ be a solution of \eqref{eq:qe} with
$v(u,v_0,0)=v_0$ and $u(u_0,v,\epsilon)$ solution of \eqref{eq:qe}
with $u(u_0,v,0)=u_0$. Then the following variational equations
holds:

\begin{equation}\label{eq:veep}\aligned
   c_\epsilon \; & + 2  v_{\epsilon u}=0,\;\;\; \;\;\;\;\; \;\;  a_\epsilon \;   + 2  u_{\epsilon v}=0,  \\
  c_{\epsilon\epsilon}& +2c_{v\epsilon} v_\epsilon -2 b_\epsilon c_\epsilon  +2  v_{u\epsilon\epsilon}=0,\\
  a_{\epsilon\epsilon}& +2a_{u\epsilon} u_\epsilon -2 b_\epsilon a_\epsilon  +2  u_{v\epsilon\epsilon}=0.
\endaligned
\end{equation}

\end{proposition}
\begin{proof}

Differentiation with respect to $\epsilon$
of \eqref{eq:qe} written as
$$a(u,v,\epsilon) (\frac{dv}{du})^2+2 b(u,v,\epsilon) \frac{dv}{du}+c(u,v,\epsilon)=0,\;\;\; v(u,v_0,0)=v_0,$$
 taking into account that
 $ \text{\small $ a_v=\frac {\partial a}{\partial v}, \;
a_\epsilon =\frac {\partial a}{\partial \epsilon },\;\;
a_{\epsilon u}= a_{u \epsilon  }=\frac {\partial^2 a}{\partial
\epsilon \partial u}=\frac {\partial^2 a}{
\partial u \partial\epsilon}, $}  $ \;
 leads to:

\begin{equation}\label{eq:ve}
   \text{\small $
(a_\epsilon +a_v v_\epsilon )  (\frac{dv}{du})^2 + 2
a\frac{dv}{du} v_{\epsilon u}+2(b_\epsilon +b_v v_\epsilon
)\frac{dv}{du}+ 2b v_{\epsilon u}+c_\epsilon +c_v v_\epsilon =0.$}
\end{equation}

Analogous notation for
$b=b(u,v(u,v_0,\epsilon),\epsilon)$,
$c=c(u,v(u,v_0,\epsilon),\epsilon)$ and   for the solution
$v(u,v_0,\epsilon)$.

Evaluation of equation \eqref{eq:ve} at $\epsilon =0 $ results in:

$$ c_\epsilon  + 2  v_{\epsilon u}=0.$$

Differentiating twice the equation  \eqref{eq:qe} and evaluating
at $\epsilon=0$  leads to:
\begin{equation}\label{eq:vee}\aligned
 c_{\epsilon\epsilon}&+ 2c_{v\epsilon} v_\epsilon +4b_\epsilon v_{\epsilon u} +2b v_{u\epsilon\epsilon}=0,\\
  c_{\epsilon\epsilon}&+2c_{v\epsilon} v_\epsilon -2 b_\epsilon c_\epsilon  +2
  v_{u\epsilon\epsilon}=0.
\endaligned
\end{equation}

Similar calculation gives the variational equations for
$u_\epsilon$ and $u_{\epsilon\epsilon}$. This ends the proof.
\end{proof}

\section{ Double Recurrence for Asymptotic Lines}\label{sc:m1}

Consider the Clifford torus ${\mathcal C}=\mathbb
S^1(\frac{1}{\sqrt{2}})\times \mathbb
S^1(\frac{1}{\sqrt{2}})\subset \mathbb S^3$ parametrized by:
\begin{equation}\label{eq:c}
 C(u,v) = \frac{\sqrt{2}}{2}( \cos(-u+v),\sin(-u+v), \cos(u+v),\sin(u+v)), \end{equation}
where $C$ is defined in the square  $Q=\{(u,v): \;\; 0\leq u\leq
2\pi, \;\; 0\leq v\leq 2\pi \}.$

\begin{proposition}\label{prop:c}  The asymptotic lines on the Clifford torus in
the coordinates given by equation \eqref{eq:c} are
given by $dudv=0$, that is, the asymptotic lines are the coordinate curves   (Villarceau circles). See Fig. \ref{fig:villa}. 

  \begin{figure}[htbp]
  \begin{center}
  \hskip 1cm
  \includegraphics[angle=0, height=3.5cm, width=5cm]{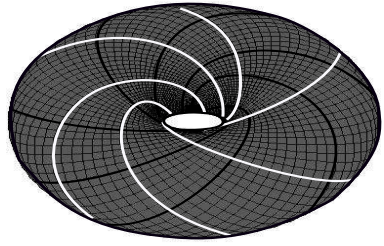}
  \caption{ Torus and  Villarceau circles  
  \label{fig:villa}}
    \end{center}
  \end{figure}

\end{proposition}

\begin{proof} The coefficients of the first fundamental form $I = E du^2+2 F dudv+
Gdv^2$ and

the second fundamental form $II = e du^2+2 f dudv+ gdv^2$ of $C$ with
respect to the normal vector field $N= C\wedge C_u\wedge C_v  $   are given by:
$$\aligned
E(u,v) =& 1,\;\;\; \;\;\;  e(u,v)= 0, \\
F(u,v)=& 0,\;\;\; \;\;\;  f(u,v)=  1,\\
G(u,v) =& 1, \;\;\; \;\;\;  g(u,v)= 0.\endaligned
$$

Therefore the asymptotic lines are defined by $dudv=0$ and so they
are the coordinate curves.
  Fig. \ref{fig:villa} is the image  of the Clifford torus
 by a stereographic projection of $\mathbb S^3$ to $\mathbb R^3$. 
\end{proof}

 \begin{theorem} \label{th:1} There are  embeddings $\alpha:  \mathbb T^2\to  \mathbb S^3$ such
that all leaves of both asymptotic foliations, $\aff$ and $\afs$,
   are dense in $\mathbb T$. See Fig. \ref{fig:dcli}

\begin{figure}[htbp]
  \begin{center}
  \hskip 1cm
  \includegraphics[angle=0, height=7cm, width=9cm]{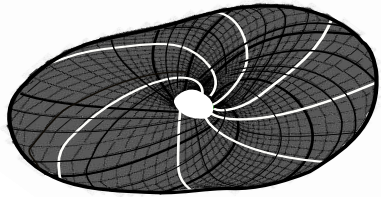}
  \caption{  Stereographic projection of a deformation of a   Clifford torus with $\epsilon=2/3$.
  \label{fig:dcli}}
    \end{center}
  \end{figure}
\end{theorem}

\begin{proof}
Let $N(u,v)=(\alpha \wedge  \alpha_u\wedge \alpha_v)  /|\alpha
\wedge  \alpha_u\wedge \alpha_v | (u,v)$ be the unit  normal
vector
to the Clifford torus.

We have that,
$$N(u,v)=\frac{\sqrt{2}}2 \big(\cos(-u+v),\sin(-u+v),-\cos(u+v),-\sin(u+v)\big).$$
Let $c(u,v)=h(u,v) N(u,v)$, $h$ being a smooth $2\pi-$ double
periodic function,  and consider for $\epsilon\ne 0 $ small the
one parameter family of embedded torus
\begin{equation}\label{eq:ceh} \alpha_\epsilon(u,v)=\frac{    C(u,v)+\epsilon c(u,v) }{|C(u,v)+\epsilon c(u,v)|}.\end{equation}
 Then the coefficients of the second fundamental
 form of $\alpha_\epsilon$ with respect to  $N_\epsilon=\alpha_\epsilon\wedge
(\alpha_\epsilon)_u\wedge (\alpha_\epsilon)_v
/|\alpha_\epsilon\wedge (\alpha_\epsilon)_u\wedge
(\alpha_\epsilon)_v|$, \;   after multiplication by  \;$1/(1+\epsilon^2 h^2)^2$,
  are given by:

\begin{equation}\label{eq:2h}\aligned
e=&    \epsilon h_{uu}+ 2 \epsilon^2 h_u h_v+ \epsilon^3(2hh_u^2-h^2 h_{uu} ),  \\
f=& 1+\epsilon h_{uv}+\epsilon^2 ( h_u^2+ h_v^2-h^2)+\epsilon^3(2hh_uh_v-h^2 h_{uv})+\epsilon^4 h^4, \\
g=&   \epsilon h_{vv}+ 2\epsilon^2 h_uh_v+\epsilon^3(2hh_v^2-h^2 h_{vv} ).
\endaligned\end{equation}

By proposition \ref{prop:eqee} the variational equations of the
implicit differential equation
\begin{equation}\label{eq:va} e(u,v,\epsilon)+2f(u,v,\epsilon)\frac{dv}{du} +g(u,v,\epsilon) (\frac{dv}{du})^2=0,\end{equation}
with $e(u,v,0)=g(u,v,0)=0$, $f(u,v,0)=1$ and $v(u,v_0,0)=v_0$ are
given by:
\begin{equation}\label{eq:vali} e_\epsilon \;   + 2  v_{\epsilon u}=0,\;\;\;
  e_{\epsilon\epsilon}  +2e_{v\epsilon} v_\epsilon -2 f_\epsilon e_\epsilon  +2  v_{u\epsilon\epsilon}=0.\end{equation}

  In fact, differentiating equation \eqref{eq:va} with respect to $\epsilon $ it is obtained:

\begin{equation} \label{eq:v1} \aligned (e_v (u,v,\epsilon) v_{\epsilon} +&e_{\epsilon} (u,v,\epsilon))+2(f_v(u,v,\epsilon) v_{\epsilon}+f_{\epsilon} (u,v,\epsilon))v_u+2f v_{u\epsilon } (u,v,\epsilon)\\
+&(g_v(u,v,\epsilon) v_\epsilon+g_\epsilon(u,v,\epsilon))(v_u)^2+ 2g(u,v,\epsilon) v_u v_{u\epsilon}=0.\endaligned\end{equation}
Making $\epsilon=0$ leads to equation $e_\epsilon \;   + 2  v_{\epsilon u}=0. $

Differentiation of  equation \eqref{eq:v1} with respect to $\epsilon$ and evaluation at  $\epsilon=0$  leads to
$$ e_{\epsilon\epsilon}  +2e_{v\epsilon} v_\epsilon -2 f_\epsilon e_\epsilon  +2  v_{u\epsilon\epsilon}=0 .$$ 

  Therefore, the integration of the linear differential equations   \eqref{eq:vali} leads to: 
\begin{equation}\label{eq:li} \text {\small $ v_\epsilon (u)= -\frac 12\int_0^u h_{uu} du,\;\;\;
v_{u\epsilon\epsilon }  =  \frac 12 h_{uuv} \int_0^u h_{uu} du \;
+\;  h_{uu}h_{uv}-2h_uh_v. $ } \end{equation}
Taking $h(u,v)=\sin^2(2v-2u)$,  it results from equation
\eqref{eq:2h} that:

\begin{equation}\label{eq:si}
e(u,v,\epsilon)=e(v,u,\epsilon)=
g(u,v,\epsilon),\;\;f(u,v,\epsilon)=f(v,u,\epsilon).
\end{equation}

In fact, from the definition of $h$ it follows that:
$$\aligned h(u,v)=& h(v,u),\; \;\;  h_u=-h_v=-2 \sin(4v-4u),\\
 h_{uu}=& h_{vv}=8\cos(4v-4u), \;\; h_{uv}=-8\cos(4v-4u).\endaligned $$

So,  a careful calculation  shows that  equation \eqref{eq:si} follows from equation \eqref{eq:2h}.

  So, from equation \eqref{eq:li}, it follows that 

$$\aligned v_\epsilon (u,v_0,0)=& - \sin(4v_0)-\sin(4v_0-4u),\\
v_{\epsilon\epsilon}(u,v_0,0)=&  -12u-4\sin( 4u)-4\sin(8v_0-4u),\\
\;&\;\;-\frac 52\sin(8v_0) + \frac{13}2\sin(8v_0-8u).\endaligned$$

Therefore,
$$  \text{\small $   v_{\epsilon}(2\pi,v_0,0)-v_{\epsilon}(0,v_0,0)= 0,\;\; v_{\epsilon\epsilon}(2\pi,v_0,0)-v_{\epsilon\epsilon}(0,v_0,0)= -24\pi$}.$$

Consider  the Poincar\'e map $\pi^1_\epsilon: \{u=0\}\to
\{u=2\pi\}$, relative to the asymptotic foliation $\aff$,
  defined by
$\pi^1_\epsilon(v_0)=v(2\pi, v_0,\epsilon)$.

Therefore,  $\pi^1_0=Id$ and it has the following expansion in
$\epsilon$:

$$\aligned \pi^1_\epsilon(v_0 )=& v_0+ \frac{\epsilon^2} 2 v_{\epsilon\epsilon}(2\pi,v_0,0) +O(\epsilon^3)\\
=& v_0-12 \pi   \epsilon^2+O(\epsilon^3)\endaligned$$
 and so the  rotation number of $\pi^1_\epsilon$ changes continuously and monotonically  with $\epsilon$.

By the  symmetry of the coefficients of the second fundamental
form in the variables $(u,v)$ and the fact that
$e(u,v,\epsilon)=g(u,v,\epsilon)$, see equation \eqref{eq:si}, it
follows that the Poincar\'e map $\pi^2_\epsilon: \{v=0\}\to
\{v=2\pi\}$, relative to the asymptotic  foliation $\afs$, defined
by
 $\pi^2_\epsilon(u_0)=u(u_0,2\pi, \epsilon)$    is conjugated to $\pi^1_\epsilon$ by an isometry.

 Therefore we can take $\epsilon \ne 0$ small such that
 the  rotation numbers of $\pi^i_\epsilon$, $i=1,2$, are, modulo $2\pi$, irrational.
  Therefore all orbits of $\pi^i$, $i=1,2$, are dense in $\mathbb S^1$.
 See  (Katok and  Hasselblatt  1995, chapter 12)   or (Palis and Melo  1982, chapter  4). This ends the proof.
\end{proof}

\section{Concluding Comments}

In this paper it  was shown that there exist   embeddings of the
torus in $\mathbb S^3$  with both asymptotic foliations having all
their leaves dense.

In (Garcia and Sotomayor 2008b) is given an example of an embedded torus in $\mathbb R^3$ with both principal  foliations having all their leaves dense.

The  technique used here  is  based on the second order
perturbation of differential equations.

It is worth mentioning
that the consideration of  only the first variational equation was
was technically  insufficient  to achieve the results of this paper.
 The same can be said for the technique of local bumpy perturbations of the Clifford Torus.

\section*{Acknowledgement}
 The authors are grateful to L.
F. Mello   for his  helpful comments. 
 
 The  authors are fellows of CNPq and
done this work under the project CNPq 473747/2006-5.

\vskip 1.5cm
 {\small
\author{\noindent Ronaldo Garcia\\Instituto de Matem\'{a}tica e Estat\'{\i}stica,\\
Universidade Federal de Goi\'as,\\CEP 74001-970, Caixa Postal
131,\\Goi\^ania, GO, Brazil}}
\vskip .5cm
 {\small\author{\noindent Jorge Sotomayor\\Instituto de 
Matem\'{a}tica e Estat\'{\i}stica,\\Universidade de S\~{a}o Paulo,
\\Rua do Mat\~{a}o 1010, Cidade Universit\'{a}ria, \\CEP 05508-090, S\~{a}o Paulo, S.P., Brazil }

\end{document}